\documentclass[11pt, oneside]{amsart}
 \usepackage[text={5.58in,8.5in},centering,letterpaper,dvips]{geometry}
\usepackage{graphicx}
\usepackage{amsfonts}
\usepackage[dvipsnames]{xcolor}
\usepackage{subcaption}
\usepackage{epsf}
\usepackage{amssymb}
\usepackage{amsmath}
\usepackage{amscd}
\usepackage{pdfpages}
\usepackage{fancyhdr}
\usepackage{setspace}
\usepackage{soul} 
\usepackage[all]{xy}
\usepackage{verbatim}
\usepackage{enumerate}
\usepackage[colorlinks=true, urlcolor=NavyBlue, linkcolor=NavyBlue, citecolor=NavyBlue]{hyperref}

\theoremstyle{theorem}
\newtheorem{theorem}{Theorem}[section]
\newtheorem{proposition}[theorem]{Proposition}
\newtheorem{lemma}[theorem]{Lemma}
\newtheorem{question}[theorem]{Question}

\makeatletter
\newtheorem*{rep@theorem}{\rep@title}
\newcommand{\newreptheorem}[2]{%
\newenvironment{rep#1}[1]{%
 \def\rep@title{#2 \ref{##1}}%
 \begin{rep@theorem}}%
 {\end{rep@theorem}}}
\makeatother

\newreptheorem{theorem}{Theorem}
\newreptheorem{lemma}{Lemma}
\newreptheorem{question}{Question}
\newreptheorem{corollary}{Corollary}
\newreptheorem{proposition}{Proposition}

\theoremstyle{definition}

\newtheorem{remark}[theorem]{Remark}

\newtheorem{example}[theorem]{Example}

\newcommand{\Z}{\mathbb{Z}}

\newcommand{\N}{\mathbb{N}}
\newcommand{\Q}{\mathbb{Q}}

\newcommand{\id}{\text{id}}
\newcommand{\Int}{\text{Int}}

\newcommand{\Cc}{\mathcal C}

\newcommand{\Ll}{\mathcal L}

\newcommand{\Rr}{\mathcal R}
\newcommand{\Ss}{\mathcal S}
\newcommand{\Tt}{\mathcal T}

\newcommand{\X}{\times}

\newcommand{\Sym}{\textsc{Sym}}
\newcommand{\Diff}{\textsc{Diff}}

\makeatletter
\def\@seccntformat#1{%
  \protect\textup{\protect\@secnumfont
    \ifnum\pdfstrcmp{subsection}{#1}=0 \bfseries\fi
    \csname the#1\endcsname
    \protect\@secnumpunct
  }%
}  
\makeatother

\begin{document}

\rhead{\thepage}
\lhead{\author}
\thispagestyle{empty}


\raggedbottom
\pagenumbering{arabic}
\setcounter{section}{0}


\title{Knots bounding non-isotopic ribbon disks}

\author{Jeffrey Meier}
\address{Department of Mathematics, Western Washington University 
Bellingham, WA 98229}
\email{jeffrey.meier@wwu.edu}
\urladdr{http://jeffreymeier.org} 

\author{Alexander Zupan}
\address{Department of Mathematics, University of Nebraska--Lincoln 
Lincoln, NE 98229}
\email{zupan@unl.edu}
\urladdr{https://math.unl.edu/azupan2}

\begin{abstract}
	We exhibit infinitely many ribbon knots, each of which bounds infinitely many pairwise non-isotopic ribbon disks whose exteriors are diffeomorphic.  This family provides a positive answer to a stronger version of an old question of Hitt and Sumners.
	The examples arise from our main result: a classification of fibered, homotopy-ribbon disks for each generalized square knot $T_{p,q}\# \overline{T}_{p,q}$ up to isotopy.
	Precisely, we show that each generalized square knot bounds infinitely many pairwise non-isotopic fibered, homotopy-ribbon disks, all of whose exteriors are diffeomorphic.  When $q=2$, we prove further that infinitely many of these disks are also ribbon; whether the disks are always ribbon is an open problem.
\end{abstract}

\maketitle

\section{Introduction}\label{sec:intro}

In this paper, we look at the problem of classifying the disks in $B^4$ bounded by a given knot in $S^3$.  Classification can take different forms:  Disks can be considered up to isotopy, up to isotopy rel boundary, or up to the diffeomorphism types of their exteriors.  In 1981, Hitt and Sumners asked the following question to draw out the differences between these notions of equivalence.\vspace{.2cm}

\begin{question}~\cite{HitSum_81}\label{q:1}
Does there exist an infinite family of non-isotopic slice disks in $B^4$ with the same exterior?
\end{question}

In the smooth category, Hitt and Sumners's question was answered affirmatively in dimension $n \geq 6$ by Plotnick~\cite{Plo_83} and $n=5$ by Suciu~\cite{Suc_85}.  Much more recently, Abe and Tange produced a collection $\{R_i\}_{i\in\N}$ of pairwise non-isotopic ribbon disks with diffeomorphic exteriors; however, these disks are distinguished by their boundary knots~\cite{AbeTan_22}.  This suggests a refined version of the question of Hitt and Sumners:

\begin{question}
Does there exist a knot $K \subset S^3$ bounding infinitely many non-isotopic disks in $B^4$ with the same exterior?
\end{question}

A positive answer also answers Question~\ref{q:1} but insists upon the stronger condition that the boundary of each disk is the same knot in $S^3$.  We answer the refined question by proving that infinitely many knots bound a family of such disks.  Rather unexpectedly, this collection of knots includes some of the very simplest ribbon knots, including the square knot, one of the two nontrivial ribbon knots with the smallest possible crossing number.

\begin{theorem}
\label{thm:main}	
Every generalized square knot $T_{p,q} \# \overline{T}_{p,q}$ bounds infinitely many pairwise non-isotopic fibered, homotopy-ribbon disks in $B^4$, all of whose exteriors are diffeomorphic.  In addition, when $q=2$, infinitely many of these disks are ribbon disks.
\end{theorem}

While it is straightforward to find knots that bound infinitely many ribbon disks that are pairwise inequivalent when considered up to isotopy rel-boundary, these examples tend to be isotopic.  For instance, every ribbon knot of the form $K \# \overline{K}$ has an infinite family of such disks.  Although distinguishing ribbon disks with the same boundary up to isotopy is rather difficult, there are some examples.
In 1991, Akbulut constructed a pair of non-isotopic ribbon disks for a fixed knot~\cite{Akb_91}, and Hayden has recently given infinitely many such pairs of knots disks~\cite{Hay_20}.  These examples have the interesting feature that the pairs of disks have exteriors that are homeomorphic, but not diffeomorphic.  In contrast, we construct infinitely many pairwise non-isotopic ribbon disks with fixed boundary whose exteriors are diffeomorphic.

Below, we state two more technical theorems that imply the first statement in Theorem~\ref{thm:main}.  The first of these theorems is taken from~\cite{MeiZup_22}, and we prove the second in Section~\ref{sec:proof}.  For $p>q>1$, let $T_{p,q}$ denote the $(p,q)$--torus knot.
Let $Q_{p,q} = T_{p,q}\# \overline{T}_{p,q}$, where $\overline T_{p,q}$ is mirror of $T_{p,q}$.
We call $Q_{p,q}$ a \emph{generalized square knot}.
In previous work~\cite{MeiZup_22}, the authors described the set of all fibered, homotopy-ribbon disks bounded by $Q_{p,q}$, up to isotopy rel-boundary.

Recall that a \emph{homotopy-ribbon disk} bounded by a knot $(S^3,K)$ is a pair $(B,D)$, where $B$ is a homotopy 4--ball $B$ and $D\subset B$ is a neatly embedded disk such that $\partial(B,D) = (S^3,K)$ and there is a surjection
$$\pi_1(S^3\setminus K)\twoheadrightarrow\pi_1(B\setminus D).$$
Homotopy-ribbon disks were first studied by Casson and Gordon, who established a beautiful connection with fibered knots~\cite{CasGor_83}.
The relevant prior work can be summarized as follows.

\begin{theorem}\cite[Theorem~1.6]{MeiZup_22}
\label{thm:MZ_disks}
	Fix a generalized square knot $Q = Q_{p,q}$.
	There exists an infinite family
	$$\Rr = \{(B_{c/d},D_{c/d}) \,\vert\, c/d\in\Q, \,\text{$c$ even}\}$$
	of fibered, homotopy-ribbon disks bounded by $(S^3,Q$) such that the following hold:
	\begin{enumerate}
		\item Every fibered, homotopy-ribbon disk $(B,D)$ bounded by $(S^3,Q)$ is isotopic rel-boundary to a member of $\Rr$.
		\item Each $B_{c/d}$ is diffeomorphic to $B^4$.
		\item The members of $\Rr$ are pairwise distinct up to isotopy rel-boundary.
		\item The members of $\Rr$ have diffeomorphic exteriors.
		\item $(B_0, D_0)$ is the product ribbon disk $(S^3, T_{p,q})^\circ\times I$.
	\end{enumerate} 
\end{theorem}

In light of part (2), we will henceforth write the set of fibered, homotopy-ribbon disks bounded by $Q$ as
$$\Rr = \{D_{c/d}\,\vert\, c/d\in\Q,\, \text{$c$ even}\}.$$
The new achievement of the present article is the following.

\begin{theorem}
\label{thm:class}
	Fix a generalized square knot $Q=Q_{p,q}$, and let $\Rr$ be the collection of all fibered, homotopy-ribbon disks bounded by $Q$.
	Two members $D_{c/d}$ and $D_{c'/d'}$ of $\Rr$ are isotopic if and only if
	$$\frac{c'}{d'} = \pm\frac{c+2nd}{d}$$
	for some $n\in\Z$.
\end{theorem}

Taken together, these two theorems yield the first claim of Theorem~\ref{thm:main}.  The proof of Theorem~\ref{thm:class} combines Theorem~\ref{thm:MZ_disks} with an analysis of the mapping class group $\Sym(S^3,Q)$ of isotopy classes of automorphisms of $S^3$ that preserve $Q$ setwise.
We remark that by Proposition~\ref{prop:equiv} below, two disks $D$ and $D'$ in $B^4$ are isotopic (resp., isotopic rel-boundary) if and only if the pairs $(B^4,D)$ and $(B^4,D')$ are diffeomorphic (resp., diffeomorphic rel-boundary).  It follows that the above classifications up to isotopy and isotopy rel-boundary also give classifications of pairs up to diffeomorphism and diffeomorphism rel-boundary.

It is an open question whether every ribbon disk bounded by a fibered knot is a fibered, ribbon disk~\cite[Question~1.4]{Mil_21}.
If this question has an affirmative answer, then $\Rr$ contains all ribbon disks bounded by $Q$.

By Theorem~\ref{thm:MZ_disks}, the homotopy-ribbon disks $D_{c/d}$ are handle-ribbon, meaning their exteriors can be built with 4--dimensional handles of index 0, 1, and 2.
(This condition is sometimes called \emph{strongly homotopy-ribbon} in the literature; see the discussion in~\cite{MilZup_20}.)
In certain cases, it can be shown that the disks are ribbon.
The simplest instance of this is the case $d=1$:
Any two disks $D_{2n/1}$ and $D_{2m/1}$ are isotopic, and $D_0$ is the product ribbon disk.
In Section~\ref{sec:background}, we connect these disks with the twist-spinning construction introduced by Zeeman~\cite{Zee_65}. 

More interesting is the following theorem, which considers the disks arising in the cases $q=2$ and $c=2$.

\begin{theorem}
\label{thm:ribbon}
	Fix a generalized square knot $Q=Q_{2k+1,2}$.
	For all $m\in\Z$, the disk $D_{2/(2m+1)}$ bounded by $Q$ is ribbon.
\end{theorem}

We derive this theorem in Section~\ref{sec:ribbon} using forthcoming work that examines the knot types appearing among the attaching links for handle-decompositions $B_{c/d}$~\cite{zupanetal}.  Combining Theorems~\ref{thm:class} and~\ref{thm:ribbon} yields the second claim in Theorem~\ref{thm:main}.  Additionally, we exhibit band presentations for the ribbon disks $D_{2/(2m+1)}$ bounded by the square knot $Q_{3,2}$; see Figure~\ref{fig:bands} and Example~\ref{ex:bands}.
Finally, we verify ribbonness for an additional, sporadic example.

\begin{proposition}
\label{prop:4/5}
	The disk $D_{4/5}$ bounded by $Q_{3,2}$ is ribbon.
\end{proposition}

The proof of this proposition is non-constructive and computer-aided, and it remains an open problem to exhibit a band presentation for this ribbon disk.  
We do not know whether the disk $D_{c/d}$ bounded by $Q_{p,q}$ is ribbon in general.  We note in Remark~\ref{rmk:alt} that an alternative proof of Proposition~\ref{prop:4/5} can be obtained by combining results in~\cite{GST} and~\cite{Sch_16}.


\subsection*{Acknowledgements}

The authors express their gratitude to the Max Planck Institute for Mathematics in Bonn, Germany for their hospitality and support during the fall of 2022, when this work began.
The first author was supported by NSF grant DMS-2006029, and the second author was supported by NSF grant DMS-2005518 and a Simons Fellowship.

\section{The disks $D_{c/d}$ and some background}\label{sec:background}

In this section, we recall the construction of the fibered, homotopy-ribbon disks $D_{c/d}$ bounded $Q=Q_{p,q}$ first described in~\cite{MeiZup_22}.  First, we explore various equivalence relations on disks in the 4--ball.

\subsection{Diffeomorphism vs. isotopy of disks}

In Section~\ref{sec:intro}, the results of this paper and prior work were phrased so as to pertain to the problem of classifying disks bounded by a knot up to the equivalence relations of isotopy and isotopy rel-boundary.
It turns out that these equivalence relations are the same as the equivalence relations induced by diffeomorphism and diffeomorphism rel-boundary of pairs.  For completeness, we state a precise version of this that is adapted to our definitions:

Let $D$ and $D'$ be neatly embedded disks in $B^4$.
	Then, the disks $D$ and $D'$ are
	\begin{enumerate}
		\item \emph{isotopic} if there is an ambient isotopy $F\colon B^4\times I \to B^4$ such that $F_1(D)=D'$;
		\item \emph{isotopic rel-boundary} if there is an ambient isotopy $F\colon B^4\times I \to B^4$ such that $F_1(D)=D'$ and $F_t\vert_{S^3}=\id_{S^3}$ for all $t\in I$;
		\item \emph{diffeomorphic} if there is a diffeomorphism $\psi\colon B^4\to B^4$ with $\psi(D)=D'$; and
		\item \emph{diffeomorphic rel-boundary} if there is a diffeomorphism $\psi\colon B^4\to B^4$ with $\psi(D)=D'$ and $\psi\vert_{S^3}=\id_{S^3}$.
	\end{enumerate}

The following is essentially Lemma~2.2 of~\cite{JuhZem_20}, phrased in terms of our definitions.

\begin{proposition}
\label{prop:equiv}
	Let $D$ and $D'$ be neatly embedded disks in $B^4$.
	Then, $D$ and $D'$ are isotopic (resp. isotopic rel-boundary) if and only if the pairs $(B^4,D)$ and $(B^4,D')$ are diffeomorphic (resp. diffeomorphic rel-boundary).
\end{proposition}

Next, we illustrate that many knots bound infinitely many ribbon disks that are pairwise non-isotopic rel-boundary.

Let $J\subset S^3$ be any knot, and let $K=J\#\overline{J}$.
Let $(B^3,J^\circ)$ be the knotted arc obtained by puncturing $S^3$ at a point in $J$.
For $t\in[0,2\pi]$, let $\tau_t\colon B^3\to B^3$ denote rotation through $2\pi t$ radians about an axis intersecting $\partial B^3$ in $\partial B^3\cap J^\circ$.

Let $(B^4, D_0) = (B^3,J^\circ)\times I$; we call $D_0$ the \emph{product ribbon disk} with boundary $(S^3,K)$.
Let $(B^4, D_n) = \Tt^n((B^3,J^\circ)\times I))$, where $\Tt(x,t) = (\tau_t^n(x),t)$.
Succinctly, $D_n$ is formed by tracing the isotopy described by twisting the knotted arc $J^\circ$ through $n$ full rotations about its end points across an interval.
We call $D_n$ the $n$--twist ribbon disk for $K$; that these disks are ribbon is a straightforward verification (see, for example, Figure 21(a) of~\cite{MZB}).

\begin{proposition}
\label{prop:rel-boundary}
	Let $J$ be any knot, and let $K=J \# \overline{J}$.
	The $n$--twist ribbon disks bounded by $K$ are pairwise isotopic but pairwise non-isotopic rel-boundary.
\end{proposition}

\begin{proof}
	First, note that $D_n$ and $D_m$ are isotopic by construction.
	Conversely, suppose $D_n$ is isotopic rel-boundary to $D_m$ for some $n,m\in\Z$.
	Then, the union $\Ss = D_n\cup_K\overline D_m$ is a knotted sphere in $S^4$ that is the double of a ribbon disk -- hence a ribbon 2--knot.
	After an isotopy supported near the equator $S^3$, we have that $\Ss$ is isotopic to the union $D_{n-m}\cup\overline D_0$.
	However, this means that $\Ss$ is the $(n-m)$--twist spin of $J$, as defined by Zeeman~\cite[Section~6]{Zee_65}.
	But Cochran has shown that a non-trivial twist spun knot cannot be ribbon~\cite[Corollary~3.1]{Coc_83}.
	It follows that $n=m$.
\end{proof}

Families of $n$--twist ribbon disks arise in the context of our constructions; see Remark~\ref{rmk:twist_disks} below.




\subsection{The construction of the disks $D_{c/d}$}
\label{subsec:construction}
\

An \emph{R-link} is an $n$--component link $L$ in $S^3$ such that there exists a Dehn surgery on $L$ yielding $\#^n(S^1 \X S^2)$.  Every R-link $L$ gives rise to a homotopy 4--ball $B_L$ obtained by attaching $n$ 2--handles to $S^3 \X I$ along $L$ and capping off the resulting $\#^n(S^1 \X S^2)$ boundary component with $n$ 3--handles and a 4--handle.

Fix $p>q>1$, and let $Q=Q_{p,q}$.
The generalized square knot $Q$ is fibered, with fiber surface $F$ of genus $g=(p-1)(q-1)$ and monodromy $\varphi\colon F\to F$.
Let $Y_Q = S^3_0(Q)$ be the result of 0--framed Dehn surgery on $Q$, so $Y_Q$ is a closed surface-bundle over $S^1$ with fiber $\widehat F = F\cup D^2$ and monodromy $\widehat\varphi = \varphi\cup\id$.
The map $\widehat\varphi$, which we refer to as the \emph{closed monodromy} of $Q$, is periodic of order $pq$, up to isotopy.
Thus, $\widehat\varphi$ generates a $\Z_{pq}$ action on $\widehat F$, and the quotient of this action a the 2--sphere $S$.
This quotient map is a $pq$--fold cyclic covering with four branched points, two of order $p$ and two of order $q$, so we regard $S$ as a pillowcase.
By cutting $\widehat F$ open along a pair of graphs, we can represent $\widehat F$ as a polygonal annulus with edge identifications; in this representation, $\widehat\varphi$ acts by rotation through $2\pi/pq$ radians.  This geometric set-up is shown in Figure~\ref{fig:curves_0} in the case of $Q = Q_{3,2}$.
We refer the reader to~\cite[Section~4]{MeiZup_22} for complete details (with additional examples in the $Q = Q_{3/2}$ case appearing in~\cite{Sch_16} and~\cite{MZ_Dehn}).

\begin{figure}[ht!]
	\centering
	\includegraphics[width=.9\textwidth]{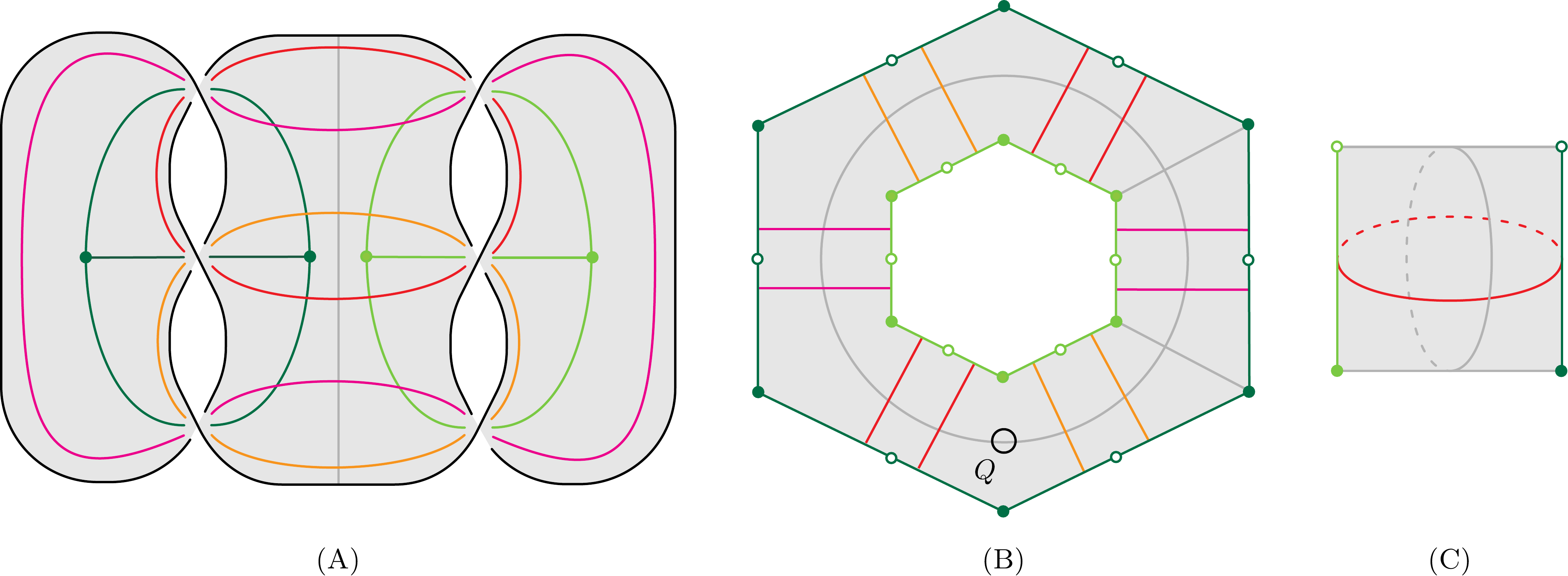}
	\caption{The fiber surface $F$ for the square knot $Q$ is shown in (\textsc A) with two embedded graphs.
	Capping off with a disk and cutting open along these graphs gives annular representation of $\widehat F = F\cup D^2$ with 12--gon boundaries shown in (\textsc B).  (In later figures, we include only the filled vertices, treating the boundary curves as hexagons.)
	Taking the quotient by the $\Z_6$ action gives the pillowcase shown in (\textsc C), where the slopes $\lambda_\infty$ and $\lambda_0$ are drawn.
	The lifts $\Lambda_\infty$ and $\Ll_0$ of $\lambda_\infty$ and $\lambda_0$ by this quotient map are drawn in $\widehat F$ in (\textsc B).
 	In (\textsc A), the link $\Ll_0$ has been isotoped to lie in $F$ in $S^3$, where three pairs of parallel curves are replaced with three curves.  The intersection of $\Lambda_\infty$ with $F$ is a separating arc.}
	\label{fig:curves_0}
\end{figure}

Given $c/d\in\Q$ with $c$ even, we can consider the slope $\lambda_{c/d}$ in $S$.
The lift of $\lambda_{c/d}$ to $\widehat F$ is a $pq$--component multi-curve $\Ll_{c/d}$.  The closed monodromy $\widehat\varphi$ cyclically permutes these curves.
The curves can be isotoped in $\widehat F$ to lie entirely within $F$.
We denote the resulting derivative link in $S^3$ by $\Ll_{c/d}$, noting that it is well-defined only up to slides over $\partial F = Q$.
For more examples of the $\Ll_{c/d}$, see Figures~\ref{fig:curves_2} and~\ref{fig:curves_3}.

\begin{remark}
\label{rmk:scharlemann}
	When $q=2$, the multi-curve $\Ll_{c/d}$ consists of $p$ pairs of parallel curves, as depicted in Figure~\ref{fig:curves_0} (\textsc B) in the the case $p=3$.
	Discarding one curve from each pair, we get a new, $p$--component multi-curve in $\widehat F$ that can be isotoped to be $\Z_{2p}$--equivariant.
	When $q=2$, we work with this $p$--component multi-curve, which continue to denote $\Ll_{c/d}$ (in an abuse of notation).
	This explains why $\Ll_0$ is shown as three curves (in shades of red) in Figure~\ref{fig:curves_0}(\textsc A).
\end{remark}

Finally, we can define the disk $D_{c/d}$:  In~\cite{MeiZup_22}, the authors show that $Q \cup \Ll_{c/d}$ is an R-link.
Let $B_{c/d} = B_{Q \cup \Ll_{c/d}}$ be the corresponding homotopy 4--ball.
We define the disk $D_{c/d}$ to be the core of the 2--handle attached along $Q$ in $B_{c/d}$, noting that $B_{c/d}$ is diffeomorphic to the standard 4--ball by Theorem~\ref{thm:MZ_disks}.

\begin{remark}
\label{rmk:twist_disks}
	Part (3) of Proposition~\ref{prop:null} below shows that, for each $p>q>1$, the disk $D_{2n/1}$ is the $n$--twist ribbon disk for $Q_{p,q}$.
	Thus, the family $\Rr$ of disks for $Q_{p,q}$ expands a classical family of ribbon disks.
	While every disk $D_{c/d}$ is homotopy-ribbon by construction, it is unknown whether $D_{c/d}$ is ribbon in general; cf. Section~\ref{sec:ribbon}.
\end{remark}

\section{The group $\Sym(S^3,Q)$}\label{sec:symmetries}

In this section we analyze the symmetries of a generalized square knot $Q=Q_{p,q}$.
Let $\Sym(S^3,Q)$ denote the symmetry group of $Q$.
This is the group of diffeomorphisms of the pair $(S^3,Q)$, modulo the normal subgroup consisting of those diffeomorphisms that are isotopic to the identity through diffeomorphisms of the pair.
We refer the reader to~\cite[Section~10.6]{Kaw_96} for complete details and a survey of classical results regarding symmetry groups of knots.
Symmetry groups of knots are generally well understood, especially in the case of prime knots.
Presently, we require a precise understanding $\Sym(S^3,Q)$ for the composite knot $Q$, the details of which we have not been able to find in the literature, so we provide them here.
The main goal of this section is to show that $\Sym(S^3,Q)$ is generated by three elements, which we will denote $\alpha$, $\beta$, and $\tau$.

Consider the Heegaard decomposition $S^3 = B_+\cup_S B_-$ of the 3--sphere into two 3--balls so that the 2--sphere $S = \partial B_+ = \partial B_-$ is a decomposing sphere for $Q$.
Write $K_\pm = B_\pm\cap Q$.

Let $\alpha\colon S^3\to S^3$ be the orientation-reversing diffeomorphism given by reflection across the 2--sphere $S$.
We assume that the set-up we are considering is symmetric with respect to $\alpha$ in the sense that $\alpha(B_\pm,K_\pm) = (B_\mp,K_\mp)$.
In particular, we have that $\alpha(Q) = Q$, though $\alpha$ reverses the orientation of $Q$.

Let $\beta\colon S^3\to S^3$ denote the rotation of $S^3$ about an unknotted axis $A$ through $\pi$ radians such that $\beta(B_\pm,K_\pm) = (B_\pm,K_\pm)$.
This is a strong-inversion of $Q$, as well as of the arcs $K_\pm$, so $\beta(Q)=Q$, but $\beta$ reverses the orientation of $Q$.
Figure~\ref{fig:sym} shows the geometric set-up in the case of $Q_{3,2}$.
	
\begin{figure}[h!]
	\centering
	\includegraphics[width=.5\textwidth]{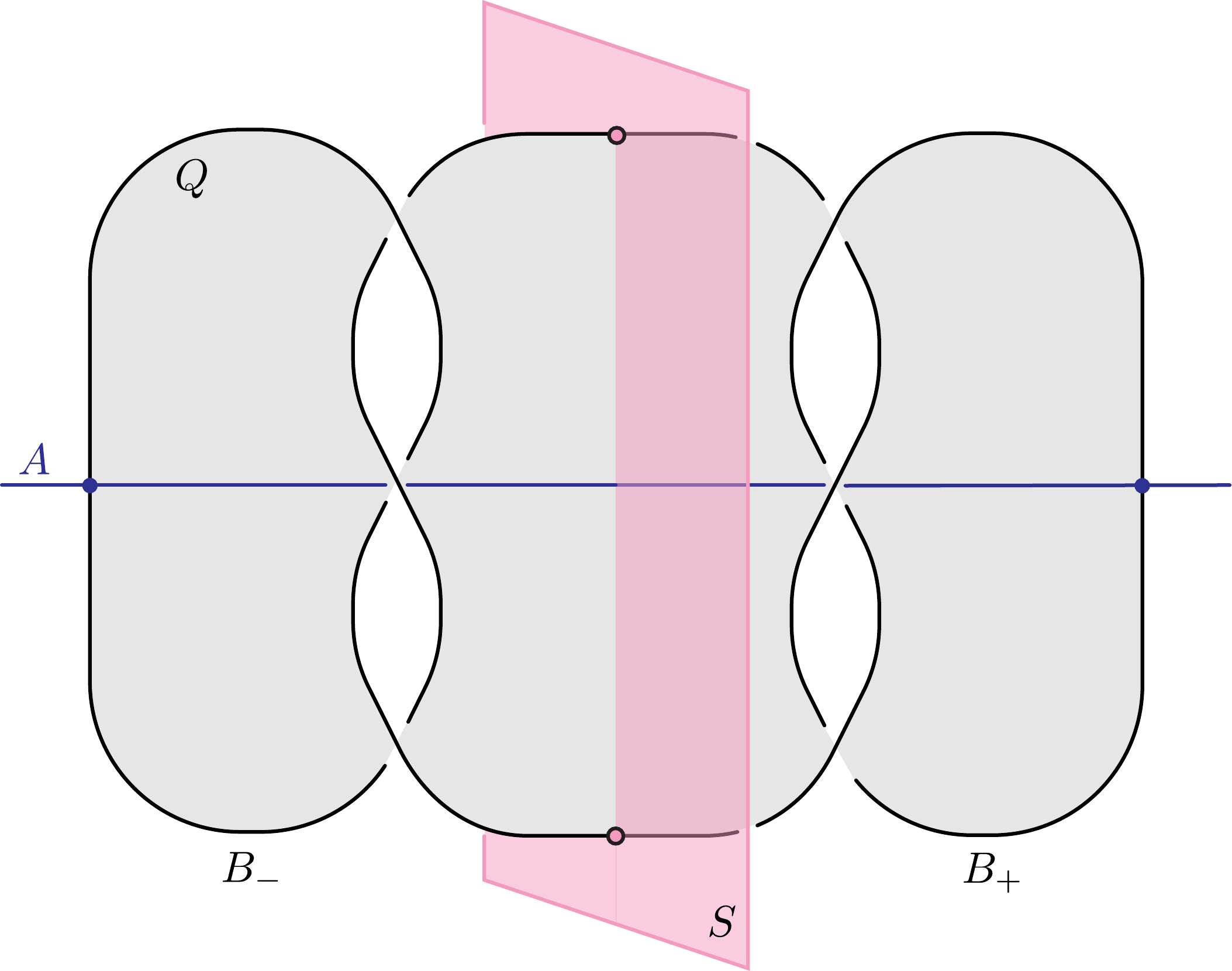}
	\caption{The square knot $Q_{3,2}$ shown with the decomposing sphere $S$ dividing $S^3$ into the balls $B_\pm$ and the strong involution axis $A$.}
	\label{fig:sym}
\end{figure}

\begin{remark}
\label{rmk:involutions}
	Note that $\alpha$ and $\beta$ are involutions in $\Sym(S^3,Q)$ and that they commute.
	The product $\alpha\circ\beta$ is a point reflection about the pair of points $S\cap A$ and is orientation-reversing on $S^3$ but preserves the orientation of $Q$.
\end{remark}

We now define the map $\tau\colon S^3\to S^3$ as a twist on the 2--sphere $S$.
Let $C = S\times[-1,1]$ be a collar neighborhood parametrized with spherical coordinates as
$$C = \left\{(\theta, \phi, t)\,|\,\theta\in[0,2\pi),\phi\in[-\pi/2,\pi/2],t\in[-1,1]\right\}$$
so that $(S \X \{t\}) \cap Q = \{(0,-\pi/2,t),(0,\pi/2,t)\}$, the poles of $S \X \{t\}$, and $S\times\{\pm1\}$ lies in $\Int(B_\pm)$.
For $x\in S^3$, define
$$\tau(x) = \begin{cases}
	x & \text{if $x\in S^3\setminus C$,} \\
	(\theta + 2\pi t+\pi,\phi,t) & \text{if $x\in C$ is given by $x=(\theta,\phi,t)$ with $t\in\left[-\frac{1}{2},\frac{1}{2}\right]$, and} \\
	x & \text{if $x\in C$ is given by $x=(\theta,\phi,t)$ with $t\in\left[-1,-\frac{1}{2}\right]\cup\left[\frac{1}{2},1\right]$.}
\end{cases}$$
Note that $\tau$ fixes the poles of $S \X \{t\}$, so $\tau$ fixes $Q$ point-wise.
As an element of $\Diff(S^3)$, the twist $\tau$ is isotopic to the identity via a family of maps that ``untwist'' $\tau$ by rotating either $B_+$ or $B_-$ along the axis spanned by the poles of $S$ given by $S\cap Q$.

Continuing, we define a pair of maps $\tau_\pm\colon S^3\to S^3$ as twists on 2--spheres parallel to $S$ in $C$ using the same parametrization as above.
For $x\in S^3$, define
$$\tau_+(x) = \begin{cases}
	x & \text{if $x\in S^3\setminus C$,} \\
	(\theta + 2\pi t,\phi,t) & \text{if $x\in C$ is given by $x=(\theta,\phi,t)$ with $t\in[0,1]$, and}\\
	x & \text{if $x\in C$ is given by $x=(\theta,\phi,t)$ with $t\in[-1,0]$.}
\end{cases}$$
The twist $\tau_-$ is defined identically, but with the domains for $t$ swapped in the last two cases.
Note that $\tau_\pm$ fixes $Q$ point-wise.
Briefly, $\tau_\pm$ is the result of isotoping $\tau$ so that the twisting occurs in $B_\pm$, which we now make precise.

\begin{lemma}
\label{lem:tau}
	The twists $\tau_+$ and $\tau_-$ are each isotopic to $\tau$ and are supported in $B_+$ and $B_-$, respectively. 
\end{lemma}

\begin{proof}
	Consider the isotopy $F\colon S^3\times [-1,1]\to S^3$ defined by
	$$F(x,s) = \begin{cases}
		x & \text{if $x\in S^3\setminus C$,} \\
		(\theta + (2\pi t + (1-s)\pi),\phi,t) & \text{if $x\in C$ is given by $x=(\theta,\phi,t)$}\\
		& \text{with $t\in\left[\frac{s-1}{2}\frac{s+1}{2}\right]$, and}\\
		x & \text{if $x\in C$ is given by $x=(\theta,\phi,t)$} \\
		& \text{with $t\in\left[-1,\frac{s-1}{2}\right]\cup\left[\frac{s+1}{2},1\right]$.}
	\end{cases}$$
	It is clear that $F(x,-1) = \tau_-(x)$, $F(x,0)=\tau(x)$, and $F(x,1)=\tau_+(x)$, as desired.
	The statement about the support of the twists is immediate.
\end{proof}

Observe that $\tau$ preserves both the orientation of $S^3$ and the orientation of $Q$.  We are now ready to prove the main technical result of the section.

\begin{proposition}
\label{prop:sym}
	The group $\Sym(S^3,Q)$ is generated by $\alpha$, $\beta$, and $\tau$.
\end{proposition}

\begin{proof}
	Let $\psi'\in\Sym(S^3,Q)$ be any element, and note that there is exactly one element $\psi\in\{\psi', \psi'\circ\alpha,\psi'\circ\beta,\psi'\circ\alpha\circ\beta\}$ such that $\psi$ is orientation-preserving on $S^3$ and preserves the orientation of $Q$.
	We will show that $\psi$ is isotopic to a power of $\tau$, which will imply that $\psi'$ is product of the proposed generators.

	In Lemma~\ref{lem:decomposing_sphere} (below), we show that $S$ is the unique decomposing sphere (up to isotopy through decomposing spheres) for $Q$.
	In light of this, we can assume (after potentially isotoping $\psi$) that $\psi(S)=S$.
	Moreover, since $\psi$ preserves the orientations of both $S^3$ and $Q$, we have $\psi(S\cap Q) = \psi(S\cap Q)$ point-wise.
	It follows that $\psi$ restricts to an automorphism of the pair $(S,S\cap Q)$.
	The mapping class $\Sym(S,S\cap Q)$ is trivial, so we can isotope $\psi$ in a neighborhood of $S$ so that $\psi$ is the identity in a neighborhood of $S$.

	Let $\psi_\pm$ denote the restriction of $\psi$ to $(B_\pm,K\pm)$, and note that $\psi_\pm$ is the identity near $\partial B_\pm = S$.
	Thus, $\psi_\pm$ is a ball-pair automorphism in the sense of Litherland~\cite[Section~1]{Lit_79}, where it was shown that the mapping class group\footnote{To be precise, Litherland worked modulo pseudo-isotopy but noted that the distinction is irrelevant in the present, 3--dimensional setting. See Remark~(1) in Section~1 and Section~6 of~\cite{Lit_79}.} $\Sym_\partial(B_\pm,K_\pm)$ is infinite cyclic when $K_\pm$ is a (punctured) torus knot~\cite[Corollary~6.4]{Lit_79}.
	The generating element of $\Sym_\partial(B_\pm,K_\pm)$ is precisely $\tau_\pm$, as defined above\footnote{Our twist $\tau$ agrees with the twist $f$ defined by Zeeman~\cite[Section~6]{Zee_65}. The generating element $\tau$ for $\Sym(B_\pm,K_\pm)$ described by Litherland is a meridional twist along the swallow-follow torus and is isotopic to Zeeman's $f$ (and our $\tau$)~\cite[Example~2.1]{Lit_79}.}.
	Since $\psi$ is the identity near $S$, we have that
	$$\psi = (\psi_+\cup\id_{B_-})\circ(\id_{B_+}\cup\psi_-),$$
	so by applying Litherland's result, we have
	$$\psi = (\tau_+^{n_+}\cup\id_{B_-})\circ(\id_{B_+}\cup\tau_-^{n_-})$$
	for some $n_\pm\in\Z$.

	By Lemma~\ref{lem:tau}, $\tau_+$ and $\tau_-$ are isotopic to $\tau$, so $\psi \simeq \tau^{n_++n_-}$, which completes the proof.
\end{proof}

We conclude this section with the proof of uniqueness of decomposing spheres for $Q$ referenced in the proof of Proposition~\ref{prop:sym}.

\begin{lemma}
\label{lem:decomposing_sphere}
	Any decomposing sphere for $Q$ is isotopic through decomposing spheres to $S$.
\end{lemma}

\begin{proof}
	Let $S$ be the decomposing sphere for $Q$ described above.
	Suppose $S'$ is another decomposing sphere for $Q$.
	After an isotopy of $S'$, we can assume that $\Cc = S\cap S'$ is a collection of simple closed curves.
	We first show that $S'$ can be isotoped to be disjoint from $S$.
	Suppose for a contradiction that $|\Cc|$ is positive and cannot be decreased by isotoping $S'$.
	Consider the set of disks in $S\setminus\Cc$.
	There are two cases: Either some disk is disjoint from $Q$ or no disk is.
	
	First, assume there is a disk $D$ in $S\setminus\Cc$ that is disjoint from $Q$.
	Let $\gamma=\partial D$.
	The curve $\gamma$ bounds two disks in $S'$; let $D'$ be the disk in $S'$ with $\partial D'=\gamma$ and satisfying $D\cap Q=\varnothing$.
	Let $P = D\cup_\gamma D'$.
	Since $P$ is an embedded sphere in exterior $S^3\setminus\nu(Q)$, which is an irreducible manifold, there is an embedded ball $B$ in $S^3\setminus\nu(Q)$ with $\partial B=P$.
	Using $B$, we can isotope $S'$, removing $\gamma$ from $\Cc$ (along with, potentially, other curves of intersection).
	This contradicts our assumption about the minimality of $\Cc$.
	Therefore, each disk of $S\setminus\Cc$ intersects $Q$; it follows that the curves of $\Cc$ are parallel and separate the two points of $S\cap Q$. 
		
	We claim $|\Cc|\leq 1$.
	To see this, note that the fact that the curves of $\Cc$ are parallel in $S$ and separate the points of $S\cap Q$ implies the same is true of the curves when viewed in $S'$.
	Let $\gamma_1$ and $\gamma_2$ be two curves of $\Cc$ that are adjacent in $S'$.
	Let $A$ be the annulus they co-bound.
	Suppose without loss of generality that $A\subset B_+\setminus\nu(K_+)$.
	This ambient space is homeomorphic to $S^3\setminus\nu(T_{p,q})$, and $\partial A$ is a pair of meridians on the boundary.
	It follows that $A$ is boundary parallel, since any essential annulus in this space intersects the boundary in two circle-fibers of the Seifert fibration, which have slope $pq$.
	Since $A$ is boundary parallel, it co-bounds a solid torus with $S$, and we can isotopy $S'$ across this solid torus to remove the curves $\gamma_1$ and $\gamma_2$ from $\Cc$.
	Thus, we assume $\Cc$ is a single curve $\gamma$.
	
	Let $S = D_1\cup_\gamma D_2$ and $S' = D_1'\cup_\gamma D_2'$, with each disk intersecting $Q$ in a single point.
	Without loss of generality, assume $D_1'\subset B^+$.
	Let $P_1 = D_1\cup_\gamma D_1'$, and let $B_1$ be the ball bounded by $P_1$ in $B_+$.
	If $B_1\cap Q$ is a trivial arc, then $D_1$ and $D_1'$ are parallel through disks in $B_1$ intersecting $Q$ in a single point, and we can isotope $S'$ across $B_1$ to make it disjoint from $S$.
	Otherwise, $B_1$ intersects $Q$ in a knotted arc.
	Let $P_2 = D_2\cup_\gamma D_1'$, and let $B_2$ be the ball bounded by $P_2$ in $B_+$.
	It cannot be the case that the arc $B_2\cap Q$ is knotted, because this would give a decomposition of $T_{p,q}$ into two non-trivial knots.
	So, $B_2\cap Q$ is unknotted, and we can proceed as above to isotope $S'$ to be disjoint from $S$.
	
	Now that $S$ and $S'$ are disjoint, we note that they bound disjoint balls $B = B_+$ and $B'\subset B_-$ in $S^3$.
	We must have that $B'\cap Q$ is a knotted arc, so it follows that $S^3\setminus(B\cup B')$ is a copy of $S^2\times I$ that intersects $Q$ in vertical arcs, and thus $S'$ is isotopic to $S$, as desired.
	
\end{proof}

\section{The proof of Theorem~\ref{thm:class}}\label{sec:proof}

In this section, we leverage the structure of $\Sym(S^3,Q)$ established in the previous section to prove our main result.
First, we analyze how extensions of the generators of $\Sym(S^3,Q)$ act on the members of $\Rr$.

\begin{proposition}
\label{prop:null}
	The generators $\alpha$, $\beta$, and $\tau$ of $\Sym(S^3,Q)$ admit extensions $\overline\alpha$, $\overline\beta$, and $\overline\tau$ to $B^4$ satisfying the following:
	\begin{enumerate}
		\item $\overline\alpha(D_{c/d}) = D_{-c/d}$
		\item $\overline\beta(D_{c/d}) = D_{-c/d}$
		\item $\overline\tau(D_{c/d}) = D_{(c-2d)/d}$
\end{enumerate}
\end{proposition}

\begin{proof}
	Recall that we have a handle-decomposition of $B^4$ relative to its boundary given by attaching 0--framed 2--handles along the link $Q\cup\Ll_{c/d}$ and capping off with 3--handles and a 4--handle.
	Let $\psi\in\{\alpha,\beta\}$.
	We claim that $\psi(Q\cup\Ll_{c/d}) = Q\cup\Ll_{-c/d}$ as framed links.
	It follows that $\psi$ can be extended to $\overline\psi$ by first extending it across the 2--handles, then across the union of the 3--handles and the 4--handle, with this last extension being guaranteed by~\cite{LauPoe_72}.
	
	To verify the claim that $\psi$ preserves the framed link $Q\cup\Ll_{c/d}$, 
	first note that $\psi$ restricts to an automorphism of $S^3\setminus\nu(Q)$ taking meridians to meridians and longitudes to longitudes.
	It follows that $\psi$ preserves the framing on $Q$ and extends to an automorphism $\widehat\psi\colon Y_Q\to Y_Q$.
	Viewing $Y_Q$ as $(S^3\setminus\nu(K_-))\cup(S^3\setminus\nu(K_+)$ (cf.~\cite[Section~6.2]{MeiZup_22}), we see $\widehat\psi$ respects the Seifert fibered structure of $Y_Q$.
	Thus, it descends to an automorphism $\psi^*$ of the pillowcase $S$.
	
	The map $\alpha^*$ is reflection of the pillowcase across $\lambda_\infty$, which sends $\lambda_{c/d}$ to $\lambda_{-c/d}$, as desired.
	(The map $\alpha^*$ was called $\varrho$ in~\cite[Lemma~4.4]{MeiZup_22}.)
	The map $\widehat\beta$ is orientation-reversing as a map from circle-fibers to circle-fibers, so the induced automorphism $\beta^*$ of the pillowcase must be orientation-reversing. 
	Any circle-fiber that intersects the axis $A$ (the fixed-point set of $\beta$) is  fixed set-wise by by $\beta$.
	The quotient $A^*$ of the axis is the fixed-point set of $\beta^*$, and it is a circle containing each cone point and intersecting the curve $\lambda_\infty$ in two points.
	Combining these facts yields that $\beta^*$ is a reflection of the pillowcase across $A^*$.
	It follows that $\beta^*(\lambda_{c/d})=\lambda_{-c/d}$, so $\widehat\beta(\Ll_{c/d}) = \Ll_{-c/d}$.
	In either case, $\psi^*(\lambda_{c/d}) = \lambda_{-c/d}$, so we have that $\psi(\Ll_{c/d}) = \Ll_{-c/d}$.
	Finally, note that $\widehat\psi(\widehat F) = \widehat F$, so $\widehat\psi$ preserves the framings on $\Ll_{c/d}\subset\widehat F$, since these framings are the surface framings.
	
	The claimed effect of $\tau$ on $\Ll_{c/d}$ was established in~\cite[Lemma~6.3]{MeiZup_22}, it was shown that the extension $\tau^*$ of $\tau$ to $Y_Q$ has the effect of
	$$\tau(\Ll_{c/d}) = \Ll_{(c-2d)/d}.$$
	(The extension $\tau^*$ was called $\Tt_\infty$ there.)
	It follows that $\tau(Q\cup\Ll_{c/d}) = Q\cup\Ll_{(c-2d)/d}$ as framed links, so we can build the extension $\overline\tau$ of $\tau$ to $B^4$ by extending it across the handle-decomposition as above.
\end{proof}

We are now ready to prove Theorem~\ref{thm:class}.

\begin{proof}[Proof of Theorem~\ref{thm:class}]
	Suppose that $D_{c/d}$ and $D_{c'/d'}$ are isotopic.
	Then by Proposition~\ref{prop:equiv}, there is a diffeomorphism $\Psi\colon B^4\to B^4$ such that $\Psi(D_{c'/d'}) = D_{c/d}$.
	Let $\psi = \Psi\vert_{\partial B^4}$.
	By Proposition~\ref{prop:sym}, we have that $\psi$ is a product of powers of $\alpha$, $\beta$, and $\tau$.
	Let $\overline\psi$ be extension of $\psi$ to $B^4$ obtained as the corresponding product of powers of the extensions $\overline\alpha$, $\overline\beta$, and $\overline\tau$.
	It follows from Proposition~\ref{prop:null} that $\overline\psi^{-1}(D_{c/d}) = D_{\pm(c-2nd)/d}$ for some $n\in\Z$.
	
	Let $\Psi' = \overline\psi^{-1}\circ\Psi$, and note that $\Psi'$ restricts to the identity on $\partial B^4$.
	Also, we have
	$$\Psi'(D_{c'/d'}) = \overline\psi^{-1}\circ\Psi(D_{c'/d'}) = \overline\psi^{-1}(D_{c/d}) = D_{\pm(c-2nd)/d}.$$
	But $\Psi'$ is a diffeomorphism rel-boundary, so by assertion (2) of Theorem~\ref{thm:MZ_disks}, we must have
	$$\frac{c'}{d'} = \pm\frac{c-2nd}{d},$$
	for some $n\in\Z$, as desired.  The reverse implication follows immediately from Proposition~\ref{prop:null}.
\end{proof}

\section{Finding ribbon disks among $D_{c/d}$}
\label{sec:ribbon}

In this section, we show that if $c=2$ and $q=2$, then any disk of the form $D_{2/2m+1}$ is ribbon.  We also show that the disk $D_{4/5}$ bounded by $Q_{3/2}$ is ribbon.  In Section~\ref{sec:intro}, we referenced another proof of this fact using results in~\cite{GST} and~\cite{Sch_16}; see Remark~\ref{rmk:alt}.
Neither proof yields an explicit construction, and the challenge of finding a ribbon presentation for this disk remains an interesting open problem.
Finally, we give ribbon presentations for the disks $D_{2/(2m+1)}$ bounded by the square knot $Q_{3,2}$.

\subsection{The proof of Theorem~\ref{thm:ribbon}}
\label{subsec:GPR}
\

Fix $p>q>1$ and $Q = Q_{p,q}$.  Recall that an $n$--component link $L$ is an R-link if Dehn surgery on $L$ yields $\#^n(S^1 \X S^2)$.  Two R-links $L$ and $L'$ are \emph{stably equivalent} if there exist unlinks $U$ and $U'$ such that $L \sqcup U$ is handleslide-equivalent to $L' \sqcup U'$, where $\sqcup$ denotes the split union of links.  In the proof of Lemma~5.1 of~\cite{MeiZup_22}, the authors showed

\begin{lemma}\label{lem:cd_equiv}
Fix $c/d \in \Q$ such that $c$ is even, and let $L$ be any nonempty sublink of $\Ll_{c/d}$.  Then $Q \cup L$ is an R-link and is stably equivalent to $Q \cup \Ll_{c/d}$ via handleslides that preserve~$Q$.
\end{lemma}

The Stable Generalized Property R Conjecture (Stable GPRC; see discussions in~\cite{GST} and~\cite{MSZ}) asserts that every R-link $L$ is stably equivalent to an unlink.  In this case, we say that $L$ is \emph{stably handleslide trivial}, and all 2-- and 3--handles in $B_L$ can be canceled after adding some additional canceling 2-- and 3--handle pairs, so $B_L$ is diffeomorphic to $B^4$.  It is unknown in general whether the R-links in Lemma~\ref{lem:cd_equiv} satisfy that Stable GPRC, but the following result, which is well-known to experts and appears in work of Abe and Tange, can be used to show that certain disks $D_{c/d}$ are ribbon.

\begin{proposition}{\cite[Lemma~5.1]{AbeTan_16}}
\label{prop:AT}
	Suppose $L$ is a stably handleslide trivial R-link.
	Then the core of each 2--handle attached along $L$ is a ribbon disk in $B^4$.
\end{proposition}

Note that our formulation of the proposition is ``upside-down'' relative to the original formulation, and we have rephrased the statement using terms relevant here.
We will also invoke

\begin{proposition}{\cite[Proposition~3.2]{GST}}\label{prop:2R}
Suppose $L$ is a 2-component R-link such one component of $L$ is an unknot.  Then $L$ is stably handeslide trivial.
\end{proposition}

In fact, Proposition 3.2 of~\cite{GST} says something slightly stronger, that such $L$ is handleslide equivalent to a 2-component unlink.  As a corollary of Lemma~\ref{lem:cd_equiv}, Proposition~\ref{prop:AT}, and Proposition~\ref{prop:2R}, we have the following sufficient condition for $D_{c/d}$ to be ribbon.

\begin{proposition}\label{prop:ribboning}
Suppose $\Ll_{c/d}$ contains an unknot.  Then $D_{c/d}$ is ribbon.
\end{proposition}

Analysis in forthcoming work~\cite{zupanetal} reveals some curious patterns in the knot types of components of $\Ll_{c/d}$.
In certain cases, the links $\Ll_{c/d}$ can be seen to contain unknotted components.
 
\begin{theorem}[\cite{zupanetal}]
\label{thm:zupanetal}
	Let $Q=Q_{2k+1,2}$ for some $k\in\N$.  For any $m\in\Z$, the link $\Ll_{2/(2m+1)}$ contains an unknotted component. 
\end{theorem}

Theorem~\ref{thm:zupanetal} and Proposition~\ref{prop:ribboning} can be combined to give a short and immediate proof of Theorem~\ref{thm:ribbon}: If $Q = Q_{2k+1,2}$, every disk of the form $D_{2/2m+1}$ bounded by $Q$ is ribbon.

\subsection{The disks $D_{2/(2m+1)}$ for $Q_{3,2}$}
\label{subsec:c=2}
\

In this subsection, we describe ribbon presentations for the ribbon disks $D_{2/2m+1}$ from Theorem~\ref{thm:ribbon} in the case $k=1$, so $Q=Q_{3,2}$.
By Proposition~\ref{prop:null}, there are involutions of $B^4$ relating $D_{c/d}$ and $D_{-c/d}$, so we restrict our attention to $m\geq 0$.  Consider the collection $\Ll_{2/1}$ of three curves in $\widehat F$, as shown in Figure~\ref{fig:curves_2}(B).
Let $V_{2/1}^1$ and $V_{2/1}^2$ denote the two curves that are disjoint from $Q$, and let $L_{2/1}$ be the union of these curves, considered as a link in $F$ in $S^3$.
It is straightforward to check that $L_{2/1}$ is an unlink, and $L_{2/1}$ cuts $F$ into a 4--punctured disk bounded by $Q$.  Thus, there is a sequence of handleslides of $Q$ over components of $L_{2/1}$ converting $Q$ to a trivial curve in $F \setminus L_{2/1}$.
By the proof of Proposition~\ref{prop:AT}, these handleslides can be used to construct a ribbon presentation for the disk $D$ bounded by $Q$ arising from the R-link $Q \cup L_{2/1}$, and by Lemma~\ref{lem:cd_equiv}, $D$ is isotopic to the disk $D_{2/1}$.  This construction yields a set $\frak b_{2/1}$ of three bands that are embedded in $F\setminus\nu(L_{2/1})$ with their endpoints attached along $Q$.
These are shown in Figure~\ref{fig:bands}(A).

\begin{figure}[ht!]
	\centering
	\includegraphics[width=.9\textwidth]{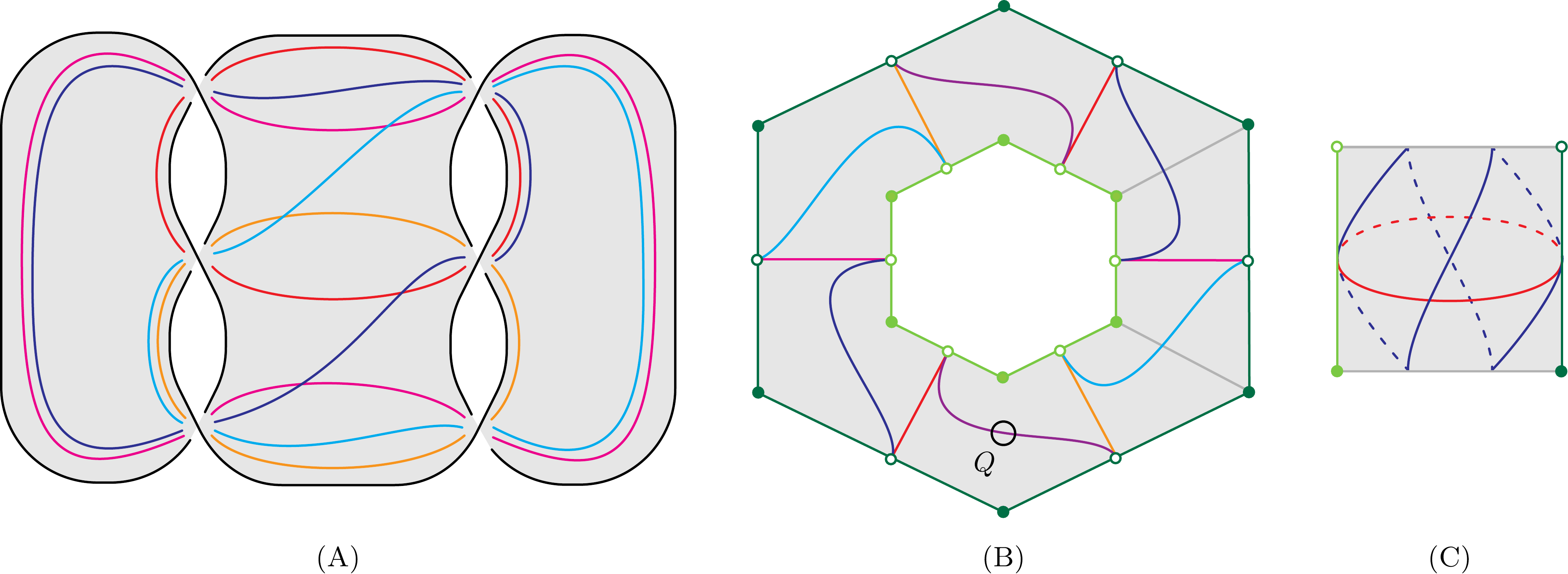}
	\caption{Obtaining the unlink derivative link $L_{2/1}$ for $Q$ in $F$.}
	\label{fig:curves_2}
\end{figure}

\begin{figure}[ht!]
	\centering
	\includegraphics[width=.9\textwidth]{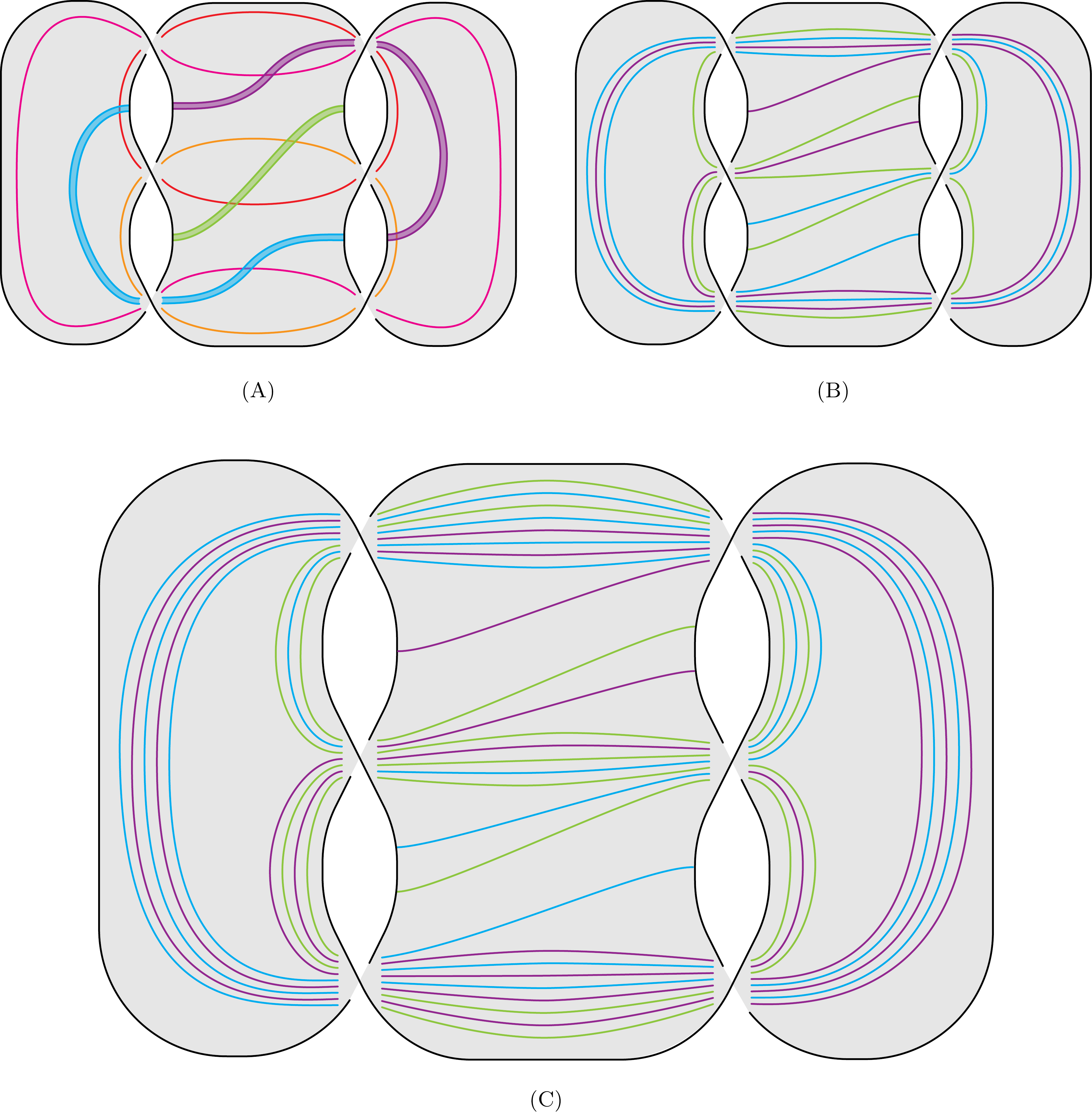}
	\caption{Ribbon presentations for the disks (\textsc A) $D_{2/1}$, (\textsc B) $D_{2/3}$, and (\textsc C) $D_{2/5}$ bounded by the square knot $Q_{3,2}$.
	In (\textsc B) and (\textsc C) only the cores of the bands are drawn; the bands are framed by the surface.}
	\label{fig:bands}
\end{figure}

Let $\tau_0$ denote a right-handed Dehn twist of $F$ along the three curves of $\Ll_0$.
Let
$$\frak b_{2/(2m+1)} = \tau_0^m(\frak b_{2/1})$$
denote the image of the bands.

\begin{proposition}
\label{prop:rib_pres}
	For any $m\geq 0$, the pair $(Q,\frak b_{2/(2m+1)})$ is a ribbon presentation for the ribbon disk~$D_{2/(2m+1)}$.
\end{proposition}

\begin{proof}
	First, note that $\tau_0$ extends to a map $\widehat\tau_0$ of $\widehat F$ and that this map satisfies $\widehat\tau_0^m(\Ll_{2/1}) = \Ll_{2/(2m+1)}$~\cite[Subsection~4.3]{MeiZup_22}.
	Let $V_{2/(2m+1)}^i = \tau_0(V_{2/1}^i)$ for $i=1,2$, and let $L_{2/(2m+1)} = V_{2/(2m+1)}^1 \cup V_{2/(2m+1)}^2$, so that that $L_{2/(2m+1)}$ is a sublink of $\Ll_{2/(2m+1)}$.
	By the proof of Theorem~\ref{thm:zupanetal}, one of the knots $V_{2/(2m+1)}^i$ is an unknot, and according to the action of the monodromy $\varphi$, $V_{2/(2m+1)}^1$ and $V_{2/(2m+1)}^2$ co-bound an annulus in $S^3$, so they are isotopic.
	As their linking number is zero, $L_{2/(2m+1)}$ is an unlink.  As above, $F \setminus L_{2/(2m+1)}$ is a 4--punctured disk bounded by $Q$, and handleslides of $Q$ over $L_{2/(2m+1)}$ in this surface convert $Q$ to a trivial curve.  These handleslides in turn yield a presentation for a ribbon disk $D$ bounded by $Q$, described as the result of attaching the bands $\frak b_{2/(2m+1)}$ to $Q$, yielding two parallel copies of the unlink $L_{2/(2m+1)}$, and capping off this 4--component unlink off with four trivial disks in $B^4$.  By Lemma~\ref{lem:cd_equiv}, this disk is isotopic to $D_{2/2m+1}$, completing the proof.
\end{proof}

\begin{example}
\label{ex:bands}
	The ribbon presentations $(Q_{3,2},\frak b_{2/1})$, $(Q_{3,2},\frak b_{2/3})$, and $(Q_{3,2},\frak b_{2/5})$ for the ribbon disks $D_{2/1}$, $D_{2/3}$, and $D_{2/5}$ bounded by $Q_{3,2}$ are shown in Figure~\ref{fig:bands}, and Proposition~\ref{prop:rib_pres} gives a recipe for writing down presentations for other disks $D_{2/(2m+1)}$.  Note that $D_{2/1}$ is the 1--twist ribbon disk for $Q_{3,2}$, so it is isotopic to the product ribbon disk $D_0$.  By Theorem~\ref{thm:class}, the ribbon presentations $(Q_{3,2},\frak b_{2/(2m+1)})$ describe an infinite family of pairwise non-isotopic ribbon disks for $Q_{3,2}$.
\end{example}

\subsection{The disk $D_{4/5}$ for $Q_{3,2}$}\

We now consider the case of the disk $D_{4/5}$ bounded by $Q=Q_{3,2}$, which we determine to be a ribbon disk.  Recall that the multi-curve $\Ll_{c/d}$ is well-defined up to isotopy in $\widehat F$, but only up to slides over $Q$ on the Seifert surface $F$.
Until now, this has not had much bearing; $\Ll_0$ has a canonical positioning in $F$, as shown in Figures~\ref{fig:curves_0}(A) and~\ref{fig:curves_2}(A), and we focused on two curves $V_{2/1}^1$ and $V_{2/1}^2$ of $\Ll_{2/1}$ in the previous subsection in a way that sidestepped the issue.
Even there, however, there was a third curve in $\Ll_{2/1}$ lacking a canonical position in $F$.
This curve is shown as intersecting $Q$ in Figure~\ref{fig:curves_2}(B) and must be pushed off of $Q$ to one side or the other to be regarded in $F$.
This issue presents itself in a more relevant way when $c/d=2/3$; see Figure~\ref{fig:curves_3}.
Let $V_{2/3}^3$ denote the component of $\Ll_{2/3}$ in $\widehat F$ that intersects the sub-surface $F$ in an arc.
This arc is shown in Figure~\ref{fig:curves_3}(C), along with the rest of $\Ll_{2/3}$.
Isotoping $V_{2/3}^3$ in $\widehat F$ to be disjoint from $Q$ gives the curve in $\widehat F$ shown in Figure~\ref{fig:45_links}(B), which will be relevant below and will still be denoted $V_{2/3}^3$.

\begin{figure}[h!]
	\centering
	\includegraphics[width=.9\textwidth]{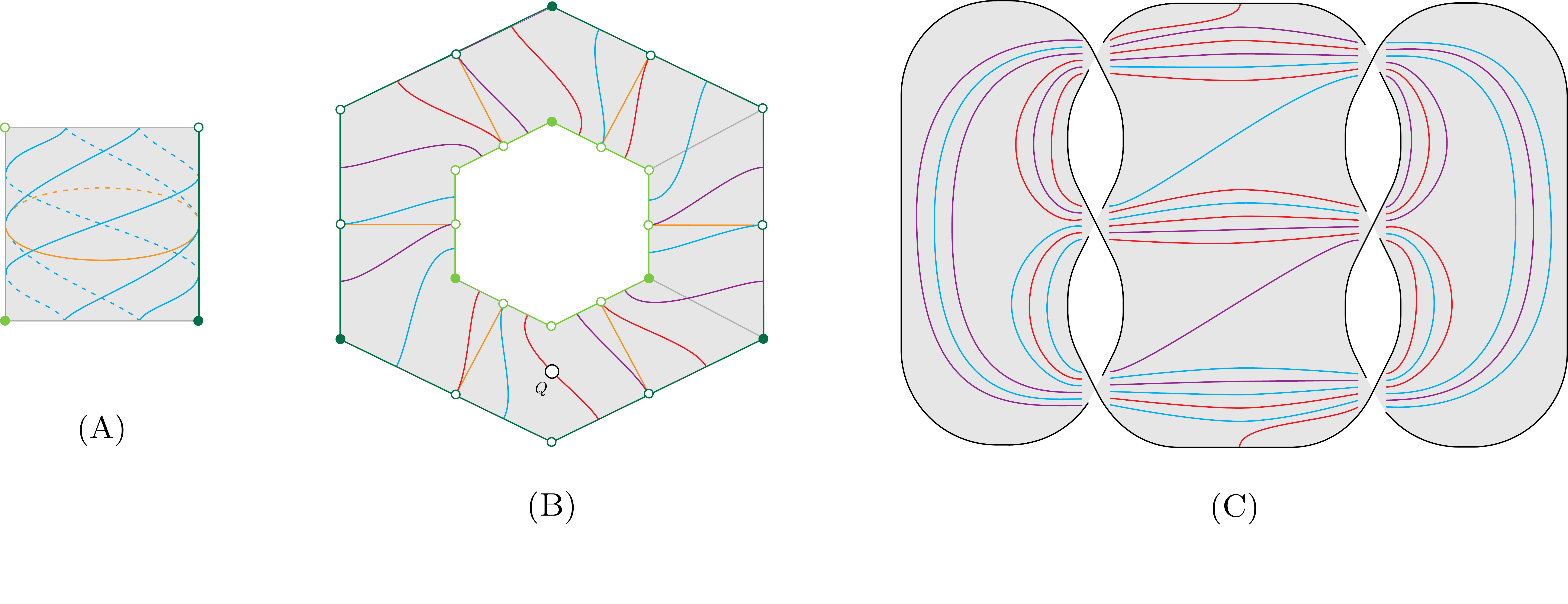}
	\caption{The intersection of the multi-curve $\Ll_{2/3}$ with the sub-surface $F$ in $\widehat F$.}
	\label{fig:curves_3}
\end{figure}

We are now ready to prove that $D_{4/5}$ is ribbon.

\begin{repproposition}{prop:4/5}
	The disk $D_{4/5}$ bounded by $Q_{3,2}$ is ribbon.
\end{repproposition}

\begin{proof}
	Let $V_{4/5}^1$ and $V_{4/5}^2$ denote the components of $\Ll_{4/5}$ in $F$ shown in Figure~\ref{fig:45_links}(A).
	(We leave to the intrepid reader the details of how these two curves in $F$ are obtained by lifting the slope $\lambda_{4/5}$ to $\widehat F$ and isotoping the resulting multi-curve so two components lie in $F$.)
	Let $L_{4/5} = V_{4/5}^1\cup V_{4/5}^2$.  By Lemma 5.1 of~\cite{MeiZup_22}, the link $Q \cup \Ll_{4/5}$ is stably equivalent to $Q \cup L_{4/5}$, \and by the discussion above, slides of $Q$ over $L_{4/5}$ convert $Q$ to a trivial curve in $F$. Thus by Proposition~\ref{prop:AT}, it suffices to show that $L_{4/5}$ is stably handleslide trivial.
	
	Now, let $V_{2/3}^3$ denote the component of $\Ll_{2/3}$ in $F$ that was discussed above and is shown in Figure~\ref{fig:45_links}(B), and consider $L_{2/3}' = V_{2/3}^3\cup Q$.  By Lemma~\ref{lem:cd_equiv}, Proposition~\ref{prop:2R}, and Theorem~\ref{thm:zupanetal}, the link $L_{2/3}'$ is stably handleslide trivial.  Finally, by Lemma~\ref{lem:isotopic} (below), $L_{4/5}$ and $L_{2/3}'$ are isotopic, and so $L_{4/5}$ is stably handleslide trivial as well.
\end{proof}

\begin{figure}[h!]
	\centering
	\includegraphics[width=.9\textwidth]{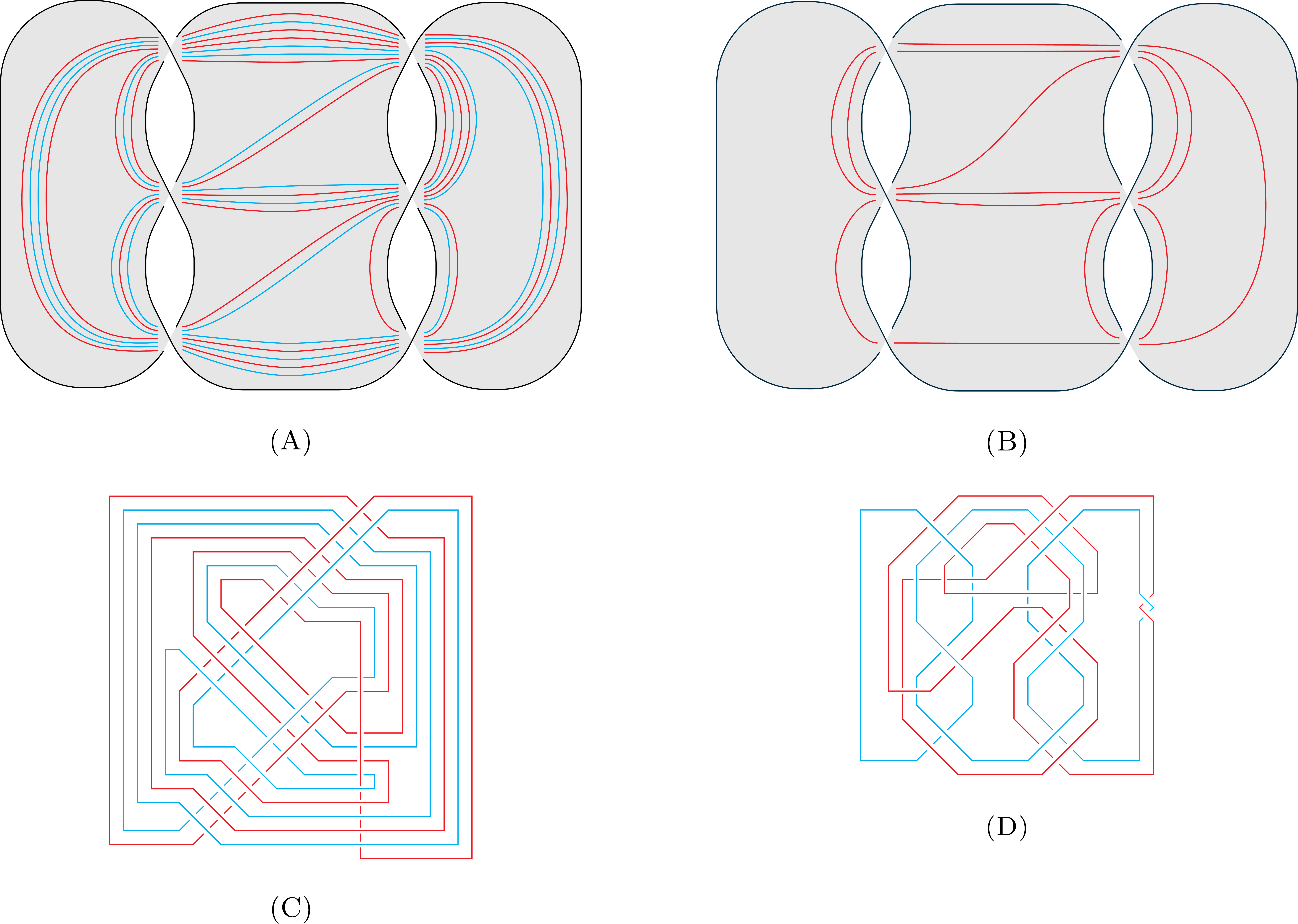}
	\caption{A pair of isotopic links. The link $L_{4/5} = V_{4/5}^1\cup V_{4/5}^2$ is shown in $F$ in (A) and isolated in space in (C). The link $L_{2/3}' = Q\cup V_{2/3}^3$ is shown with $V_{2/3}^3$ lying in the Seifert surface for $Q$ in (B) and isolated in space in (D).}
	\label{fig:45_links}
\end{figure}

We note in passing that $V_{2/3}^3 = V_{4/5}^1$ is a the stevedore knot $\mathbf{6_1}$.
The final ingredient is a computer-aided verification using SnapPy~\cite{snappy} that the links $L_{4/5}$ and $L_{2/3}'$ are isotopic.

\begin{lemma}
\label{lem:isotopic}
	The links $L_{4/5}$ and $L_{2/3}'$ are isotopic.
\end{lemma}

\begin{proof}
	Let $M_{4/5}$ and $M_{2/3}'$ denote the exteriors of $L_{4/5}$ and $L_{2/3}'$, respectively.
	We describe these links to SnapPy using the Plink editor, and SnapPy verifies that both $M_{4/5}$ and $M_{2/3}'$ are hyperbolic with volume approximately 13.5521190893.
	In fact, using the command
	\begin{center}
		\verb|M.is_isometric_to(M', return_isometries = True)|
	\end{center}
	we find that there is a unique isometry between these hyperbolic manifolds and that it preserves meridians.
	It follows that the links are isotopic.
\end{proof}

\begin{remark}
\label{rmk:alt}
An alternative proof of this fact follows from the sketch in~\cite{GST} that the link $L_{2,1}$ (in that paper) is handleslide trivial, combined with Theorem 4.1 of~\cite{Sch_16} asserting that $L_{2,1}$ is isotopic to one of the links of the form $Q \cup V^i_{2/5}$.
\end{remark}


\newpage

\bibliographystyle{amsalpha}
\bibliography{Q_disks.bib}

\end{document}